\documentclass[10pt, a4paper, journal, twocolumn]{IEEEtran}

\IEEEoverridecommandlockouts 

\usepackage{amsmath}
\usepackage{amsthm} 
\usepackage[hidelinks,colorlinks=false]{hyperref}
\usepackage{graphicx}
\usepackage{amssymb}
\usepackage{bbm}
\usepackage{easylist}
\usepackage{fp}
\usepackage{siunitx}
\usepackage{dsfont}
\usepackage{graphicx}  
\usepackage{tikz}
\usepackage[lofdepth,lotdepth]{subfig}
\usetikzlibrary{shapes,arrows}
\usepackage{tkz-berge}
\usetikzlibrary{calc}
\usetikzlibrary{fit}
\usepackage{verbatim}
\usepackage{setspace}

\tikzstyle{block} = [draw, fill=white, rectangle, 
    minimum height=3em, minimum width=6em]
\tikzstyle{sum} = [draw, fill=white, circle, node distance=1cm]
\tikzstyle{input} = [coordinate]
\tikzstyle{output} = [coordinate]
\tikzstyle{pinstyle} = [pin edge={to-,thin,black}]

\newcommand{\bal}[1] {\ensuremath{\left(\begin{array}{#1}}}
\newcommand{\ear} {\ensuremath{\end{array}\right)}}

\newcommand{\bals}[1] {\ensuremath{\left[\begin{array}{#1}}} 
\newcommand{\ears} {\ensuremath{\end{array} \right] }} 

\DeclareMathOperator{\im}{im}

\DeclareMathOperator{\inte}{int}

\DeclareMathOperator{\diag}{diag}

\newcommand{\one} {\ensuremath{\mathds{1} }} 

\let\leq\leqslant
\let\geq\geqslant

\newcommand{\calD}{\ensuremath{\mathcal{D}}}
\newcommand{\calE}{\ensuremath{\mathcal{E}}}

\newcommand{\calG}{\ensuremath{\mathcal{G}}}
\newcommand{\calH}{\ensuremath{\mathcal{H}}}

\newcommand{\calL}{\ensuremath{\mathcal{L}}}

\newcommand{\calS}{\ensuremath{\mathcal{S}}}

\newcommand{\calU}{\ensuremath{\mathcal{U}}}
\newcommand{\calV}{\ensuremath{\mathcal{V}}}
\newcommand{\calW}{\ensuremath{\mathcal{W}}}
\newcommand{\calX}{\ensuremath{\mathcal{X}}}

\newcommand{\bmat}{\begin{matrix}}
\newcommand{\emat}{\end{matrix}}
\newcommand{\bbm}{\begin{bmatrix}}
\newcommand{\ebm}{\end{bmatrix}}
\newcommand{\bpm}{\begin{pmatrix}}
\newcommand{\epm}{\end{pmatrix}}
\newcommand{\bse}{\begin{subequations}}
\newcommand{\ese}{\end{subequations}}
\newcommand{\beq}{\begin{equation}}
\newcommand{\eeq}{\end{equation}}
\newcommand{\ben}{\begin{enumerate}}
\newcommand{\een}{\end{enumerate}}
\newcommand{\beni}{\renewcommand{\labelenumi}{\roman{enumi}.}
\renewcommand{\theenumi}{\roman{enumi}}\begin{enumerate}}
\newcommand{\eeni}{\end{enumerate}\renewcommand{\labelenumi}{\arabic{enumi}.}
\renewcommand{\theenumi}{\arabic{enumi}}}
\newcommand{\bena}{\renewcommand{\labelenumi}{\alpha{enumi}.}
\renewcommand{\theenumi}{\alpha{enumi}}\begin{enumerate}}
\newcommand{\eena}{\end{enumerate}\renewcommand{\labelenumi}{\arabic{enumi}.}
\renewcommand{\theenumi}{\arabic{enumi}}}
\newcommand{\bit}{\begin{itemize}}
\newcommand{\eit}{\end{itemize}}

\newcommand{\R}{\ensuremath{\mathbb R}}

\tikzstyle{vertex}=[circle,fill=black!5,draw=black,minimum size=8mm]

\tikzstyle{terminal vertex} = [circle,fill=black!5,draw=black,minimum size=8mm]
\tikzstyle{edge} = [draw,thick,-]
\tikzstyle{edge2} = [draw,thick,->,red!50]
\tikzstyle{weight} = [font=\small]
\tikzstyle{selected edge} = [draw,line width=2pt,-,red!50]
\tikzstyle{ignored edge} = [draw,line width=5pt,-,black!20]
\tikzstyle{empty vertices} = [circle,fill=white!5]

\newtheorem{theorem}{Theorem}

\newtheorem{lemma}[theorem]{Lemma}

\newtheorem{corollary}[theorem]{Corollary}

\newtheorem{example}{Example}
\newtheorem{remark}{Remark}

\title{Optimal Weight Allocation of Dynamic Distribution Networks and Positive Semi-definiteness of Signed Laplacians}

\author{Jieqiang Wei, Alexander Johansson, Henrik Sandberg, Karl H. Johansson and Jie Chen
\thanks{*This work is supported by Knut and Alice Wallenberg Foundation, Swedish Research Council, and Swedish Foundation for Strategic Research.}
\thanks{J. Wei, A. Johansson, H. Sandberg, K. H. Johansson are with the Department of Automatic control, School of Electrical Engineering and Computer Science, 
 KTH Royal Institute of Technology,
 SE-100 44 Stockholm, Sweden.
 {\tt\small \{jieqiang, alexjoha, hsan, kallej\}@kth.se}.
Jie Chen is with the Department of Electronic Engineering, City University of Hong Kong,
Hong Kong, China.
 {\tt\small \{jichen\}@cityu.edu.hk}
}}

\begin{document}

\maketitle






\begin{abstract}
In this paper, we consider the robustness of a basic model of a dynamical distribution network. In the first problem, i.e., optimal weight allocation, we minimize the $\calH_\infty \text{- norm}$ of the dynamical distribution network subject to allocation of the weights on the edges. It is shown that this optimization problem can be formulated as a semi-definite program. Next we consider the semi-definiteness of the weighted graph Laplacian matrix with negative weights on the edges. A necessary and sufficient condition, using the effective resistance matrix, is established to guarantee the positive semi-definiteness of the Laplacian matrix. Furthermore, the bounded real lemma is derived for state-space symmetric systems.
\end{abstract}

\begin{IEEEkeywords}
Network Analysis and Control, $\calH_\infty$ control, optimization, signed Laplacian.             
\end{IEEEkeywords}

\IEEEpeerreviewmaketitle

\section{Introduction}

Modern societies critically rely on distribution networks of various kinds.
Typically, a distribution network is depicted as a graph where resources can enter the network via supply vertices and leave the network via demand vertices, together with edges that connect the supply, demand and additional internal vertices. Often, flow capacity constraints and cost functions are assigned to the edges.

Distribution networks can be divided into two classes, depending on whether the vertices can store resources or not. If the vertices can only distribute resources without storage, we refer to this type of distribution networks as static. The study of static distribution networks is a broad research topic which has a long history and a large number of applications \cite{Aronson1989}. One celebrated result is the max-flow min-cut theorem \cite{Ford1956}. The static distribution problem is closely related to monotropic programming problems which enjoy a complete and symmetric duality theory \cite{rockafellar1984network}. 

Differently from static distribution networks, in dynamical distribution networks vertice can have storage of resources. This type of models has many applications in, e.g., communication networks \cite{Ephremides1989,Segall1977}, transportation networks \cite{Moreno1995,Como2015,Lovisari2014,Moss1982}, hydraulic networks \cite{SCHOLTEN2017}, flow networks \cite{DANIELSON2013,wei2013}, and inventory and production systems \cite{Bertsimas2006,Blanchini00}. 

In this paper, we analyze the robustness of a basic dynamical distribution networks where we assign a set of single integrators to the vertices (with state variables corresponding to storage). All the integrators are controlled by the flows on the edges. On each edge, the flow is the weighted storage difference of the adjacent vertices.  Furthermore, unknown in/outflows may enter or leave the network through some of the vertices. The aim here is to minimize the induced $\calL_2$ gain from the in/outflows to the output of the network by allocating the weights on the edges, which will be called optimal weight allocation problem in this paper. The results of this problem can be relevant when designing robust multi-agent systems. Especially, our setup is similar to the setting in \cite{rai2012}, when one considers the in/outflows as malicious attacks whose goal is to maximize the differences of the storages of the vertices. Then by solving the optimal weight allocation problem, the effect of the worst attack will be minimized. The distribution networks considered in this paper can be seen as linear time-invariant port-Hamiltonian systems \cite{schaftSIAM}, but also resides in the category of state-space symmetric systems \cite{Nagashio2005,Qiu1996,WILLEMS1976,Yang2001}. One useful property of the state-space symmetric system is that its $\calH_\infty \text{- norm}$ is attained at the zero frequency \cite{TAN2001}.

One closely related problem to the optimal weight allocation, where the connection will be clear in the primary part of the paper, is the positive semi-definiteness of weighted Laplacian with both negative and positive weights. This problem is of salient importance in distributed algorithms \cite{Altafini2013,Xia2016,Xiao2003}.  This problem was considered by many authors. In \cite{Zelazo2014}, the authors provided one sufficient and necessary condition, using effective resistance, for a special weighted graph, namely those where the negatively weighted edges are isolated in different cycles in the graphs spanned by the positive edges. Under the same assumption, the authors of \cite{ChenY2016} re-derived the result in \cite{Zelazo2014} by using geometrical and passivity-based approaches. For general weighted graphs, one sufficient and necessary condition was proposed in \cite{ChenW2016,ChenW2017} using pseudo-inverse of weighted (with negative ones) Laplacian. Here we propose a sufficient and necessary condition using the effective resistance matrix of the positive subgraph from $\calH_\infty$ approach.

The contributions of this paper are listed as follows. First, we derive a bounded real lemma type of result for state-space symmetric systems. Second, the problem of minimizing the $\calH_\infty \text{- norm}$ of the dynamical distribution networks subject to the allocation of the flow capacities is formulated as a semi-definite program. Third, we present a necessary and sufficient condition of positive semi-definiteness of weighted Laplacians, with negative and positive weights, i.e., signed Laplacians.

The structure of the paper is as follows. Some preliminaries will be given in Section \ref{s:preli}. The considered class of dynamic distribution networks and the corresponding weights allocation problem, and the problem of positive semi-definiteness of weighted Laplacian are formulated in Section \ref{ProbForm}. The main results are presented in Section \ref{distnet} and \ref{s:Positive-Laplacian}. Conclusions and future work are given in Section \ref{Concl}.

The notations used in the current paper are collected as follows.

\textbf{Notation.}  A positive definite (positive semi-definite) matrix $M$ is denoted as $M\succcurlyeq 0$ ($M\succ 0$). 
The element on the $i^{\text{th}}$ row and $j^{\text{th}}$ column of a matrix $M$ is denoted $M_{ij}$. The pseudo-inverse of $M$ is $M^\dagger$. Recall that, for any finite dimensional square matrix $M$, the induced $\ell_2$ norm, denoted by $\|M\|_2$, is the largest singular value which is denoted by $\bar{\sigma}(M)$. The image of a matrix $M$ is $\im M$. The identity matrix is denoted as $I$. The vectors $\one_n$ represents a $n$-dimensional column vector with each entry being $1$.
We will omit the subscript $n$ when no confusion arises. The Euclidean norm of a vector $x$ is denoted as $\|x \|_2$. Given a set $\calS$, $\inte\{S\}$ denotes its interior.

\section{Preliminaries}\label{s:preli}

In this section, we briefly review some essentials about graph theory \cite{Bollobas98} and robust analysis \cite{zhou1998essentials}.

\subsection{Graph Theory}\label{GraphT}

An undirected graph $\mathcal{G}=(\mathcal{W},\mathcal{V},\mathcal{E})$ consists of a finite set of vertices $\mathcal{V}=\{v_1,...,v_n\}$, a set of edges $\mathcal{E}=\{ \mathcal{E}_1,...,\mathcal{E}_m\}$ that contains unordered pairs of elements of $\calV$, and a set of corresponding edge weights $\mathcal{W}=\{w_1,...,w_m\}$. 
The set of neighbours to vertice $i$ is 
\begin{equation*}
N_i=\{v_j|(v_i,v_j)\in \mathcal{E} \}.
\end{equation*}
The graph Laplacian  $L\in \mathbb{R}^{n \times n}$ is defined component-wise as
\begin{equation*}\label{Laplacian}
 L_{w,ij} =
  \begin{cases}
    \sum_{ j\in N_i} w_{ij} & \quad \text{if }   i=j,  \\
    -w_{ij}  & \quad \text{if } j\in N_i \setminus \{i\}\\
    0 & \quad \text{if } j\notin N_i,
  \end{cases},
 \end{equation*}
where both positive and negative weights are allowed. Given an arbitrary orientation for each edge, the incidence matrix $B\in~\R^{n \times m}$ is defined as
\begin{equation*}\label{incid}
B_{ij}=\begin{cases}
    1 & \quad \text{if }   \mathcal{E}_j \ \text{starts in vertice} \ v_i,  \\
   -1  & \quad \text{if }   \mathcal{E}_j \ \text{ends in vertice} \ v_i,  \\
    0  & \quad \text{else}. 
  \end{cases}
\end{equation*}
These two matrices are related by $L_w=BWB^T$, where $W=~\diag(w_1,...,w_m)$. If $W \succcurlyeq 0$, i.e., there are only positive edges, then it is well-known that the eigenvalues of $L_w$ can be structured as
$
0=\lambda_1\leq\lambda_2\leq...\leq \lambda_n, 
$
where the eigenvector corresponding to $\lambda_1=~0$ is $\mathds{1}$. If $W = I$, the Laplacian is denoted without subscript as $L$.

If a graph $\calG = (\calV,\calE,\calW)$ has both positive and negative weights, we separate the edges set $\calE$ into  $\calE_+$ and $\calE_-$, which contains the positive and negative edges, respectively. Thus $\calE=\calE_-\cup \calE_+$. Accordingly, the weight matrix $W = W_+ - W_-$, where $W_+$ ($W_-$) is the absolute value of the weights corresponding the positive (negative) edges. The Laplacian matrix is referred to as signed Laplacian which can be decomposed as 
\begin{equation*}
L_w =L_{w+} - L_{w-} : = B_{+}W_{+}B_{+}^T - B_{-}W_{-}B_{-}^T ,
\end{equation*}
where $B_{+}$ and $B_{-}$ are incidence matrices corresponding to the positive and negative sub-graphs, respectively. 

The undirected and connected graph without self-loops and with only positive weights can be associated with electrical networks \cite{Klein93}. One important concept is the \emph{effective resistance matrix}, see e.g., \cite{bullo2014,Klein93}, which is defined as  
\begin{equation*}
\Gamma = B^TL_w^{\dagger}B,
\end{equation*}
where $L_w^{\dagger}$ is the Moore-Penrose pseudo-inverse of $L_w$ and $B$ is the incidence matrix.


\subsection{$\mathcal{L}_2$-Norm and induced $\mathcal{L}_2$-Gain}

In this subsection, we recall some definitions from robust control. The notations used in this paper are fairly standard and are consistent with \cite{zhou1998essentials}, \cite{RANTZER2015}.
The space of square-integrable signals $f:[0,\infty)\rightarrow \mathbb{R}^n$ is denoted by $L_2[0,\infty)$.
For the linear time-invariant system 
\begin{align}\label{e:linear-sys}
\dot{x} & = Ax+Bu, \\ \nonumber
y & = Cx+Du,
\end{align}
the transfer matrix is $\mathbb{G}(s)=C(sI-A)^{-1}B+D$, which has the impulse response
\begin{equation*}
g(t)~=\calL^{-1}\{\mathbb{G}(s)\}~=~Ce^{At}B\mathbf{1}_+(t)+D\delta(t),
\end{equation*}
where $\delta(t)$ is the unit impulse and $\mathbf{1}_+(t)$ is the unit step defined as 
\begin{align}\nonumber
\mathbf{1}_+(t) = \begin{cases}
1, t\geq 0,\\
0, t<0.
\end{cases}
\end{align}  
If $x(0)=0$, then we have
$
y(t) = \int_{0}^t g(t-\tau) u(\tau) d\tau.
$
Then the induced $\mathcal{L}_2 \text{- gain}$ is defined as
\begin{align}\nonumber
\|g\|_{2-ind} = \sup_{u\in L_2[0, \infty)}\frac{\|y\|_2}{\|u\|_2} = \sup_{u\in L_2[0, \infty)}\frac{\|g*u\|_2}{\|u\|_2},
\end{align}
where 
$\|u(t)\|_2 =\Big(\int_{0}^{\infty}|u(t)|_2^2 dt\Big)^{\frac{1}{2} }.$

This induced $\mathcal{L}_2 \text{- gain}$, i.e., $\|g\|_{2-ind}$ or $\|\mathbb{G}\|_{2-ind}$, is often called the $\calH_\infty \text{- norm}$, denoted as $\|\mathbb{G}\|_\infty$.  It is well-know that for stable systems we have that $\|\mathbb{G}\|_\infty= \sup_{\omega\in\mathbb{R}}\bar{\sigma}\{\mathbb{G}(j\omega)\}$,
where $\bar{\sigma}(A)$ denotes the largest singular value of the matrix $A$.


If the matrices in \eqref{e:linear-sys} satisfy $A=A^\top$, $D=D^\top$ and $C=B^\top$, the system is referred to as state-space symmetric system \cite{TAN2001}. Here, we present a bounded real lemma type of result with respect to state-space symmetric system. Notice that for internally positive LTI system, the bounded real lemma was established in \cite{Tanaka2011}.
\begin{lemma}\label{lm:bounded-real-ssss}
Consider any state-space symmetric system \eqref{e:linear-sys} with ${A \prec 0}$ and $D=0$. Then the following conditions are equivalent,
\begin{enumerate}
\item $\|\mathbb{G}\|_\infty \leq \gamma$,
\item the inequality 
\begin{align}\label{e:Hinf-Riccati}
PA + AP + BB^\top+ \frac{1}{\gamma^2}PBB^\top P \preccurlyeq 0
\end{align}
has a solution $P=\gamma I$.
\end{enumerate} 
\end{lemma}

The proof is given in Appendix.

\begin{remark}
In the previous lemma, we gave an explicit solution for the Riccati inequality \eqref{e:Hinf-Riccati} where the bounded real lemma can only guarantee the existence of the solutions.
\end{remark}

\section{Problems Formulation}\label{ProbForm}

%
%
%

In this paper, we first consider the weight allocation problem in the scenario of dynamical distribution networks, which is defined on a graph $\calG=(\calV,\calE)$ with $|\calV|=n$ and $|\calE|=m$. Consider the dynamic model 
\begin{align}\label{e:plant}
\dot{x}(t) & = B u(t) + E d(t),
\end{align}
where $B$ is the incidence matrix of the graph $\calG$, $x\in\R^n$ is the system state whose components represent the storage levels in the vertices, $u\in\R^m$ is the controlled flows on the edges, $E\in\R^{n\times k}$ is an assigned matrix and $d(t)\in\R^k$ is an unknown external in/outflows. Here we assume that the image of $E$ is a subset of the image of $B$, i.e., the inflow is equal to the outflow. To simplify the composition, we further assume that, for all $i = 1,\ldots,k$, the $i$th column of $E$ consists of one element which is $\alpha_i>0$ (inflow) and one element $-\alpha_i$ (outflow), while the rest of the elements are zero. Without specification, we set $\alpha_i=1$.  A \emph{port} is a set of vertices (terminals) to where the external flows which enter and leave the network sum to zero. Thus, $E$ defines $k$ ports. One example of system \eqref{e:plant} is depicted as Fig. \ref{figure main ex1}. 

The condition $\im E\subset \im B$ is a standard assumption, in order to have a stable distribution network, for example in \cite{Blanchini00} which is recalled in the following remark.

\begin{remark}
In \cite{Blanchini00}, the authors considered a distribution network with constraints on the storages, flows and external in/outflows as
\begin{align*}
x(t) & \in \calX := \{ x\in\R^n \mid x^-\leq x \leq x^+ \} \\
u(t) & \in \calU := \{ u\in\R^m \mid u^-\leq u \leq u^+ \} \\
d(t) & \in \calD := \{ d\in\R^k \mid d^-\leq d \leq d^+ \} 
\end{align*}
where $x^+,x^-,u^+,u^-,d^+,d^-$ are assigned vectors and the inequalities hold component-wisely. First, it was proved that
the existences of a state-feedback control $u(t)\in\calU$ and a set of initial conditions $\calX_0\subset \calX$, such that for every $x(0)\in\calX_0$, the solutions of \eqref{e:plant} satisfy
\begin{align*}
x(t)\in\calX, \quad \forall d(t)\in\calD, \quad t\geq 0
\end{align*}
if and only if $E\calD\subset -B\calU$.
Then it was proved that for any $\bar{x}\in\calX$, the existence of a state feedback control law $u$, for system \eqref{e:plant}, such that $x(t)\in\calX$, $u(t)\in\calU$, and 
\begin{align*}
\lim_{t\rightarrow \infty} x(t) = \bar{x}, \quad \forall d(t)\in\calD
\end{align*}
if and only if 
\begin{align*}
E\calD \subset -\inte\{B\calU\}.
\end{align*}
It can be seen that one necessary condition to have $E\calD\subset -B\calU$ and $E\calD \subset -\inte\{B\calU\}$ is $\mathds{1}^\top E = 0$ which implies that the image of $E$ is a subset of the image of $B$ for connected graphs. 
\end{remark}

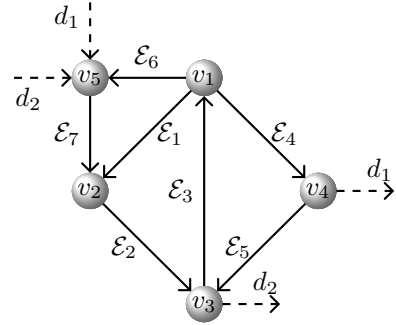
\begin{figure}[ht]
\begin{center}
\begin{tikzpicture}
\tikzstyle{EdgeStyle}    = [thin,double= black,
                            double distance = 0.5pt]
\useasboundingbox (0,0) rectangle (4cm,4.5cm);
\tikzstyle{VertexStyle} = [shading         = ball,
                           ball color      = white!100!white,
                           minimum size = 20pt,%
                           inner sep       = 1pt,]
\Vertex[style={minimum
size=0.2cm,shape=circle},LabelOut=false,L=\hbox{$v_1$},x=1.5cm,y=3.5cm]{v1}
\Vertex[style={minimum
size=0.2cm,shape=circle},LabelOut=false,L=\hbox{$v_2$},x=0cm,y=2cm]{v2}
\Vertex[style={minimum
size=0.2cm,shape=circle},LabelOut=false,L=\hbox{$v_3$},x=1.5cm,y=0.5cm]{v3}
\Vertex[style={minimum
size=0.2cm,shape=circle},LabelOut=false,L=\hbox{$v_4$},x=3cm,y=2cm]{v4}
\Vertex[style={minimum
size=0.2cm,shape=circle},LabelOut=false,L=\hbox{$v_5$},x=0cm,y=3.5cm]{v5}

\draw
(v1) edge[->,>=angle 90,thick]
node[right]{$\calE_1$} (v2)
(v2) edge[->,>=angle 90,thick]
node[left]{$\calE_2$} (v3)
(v3) edge[->,>=angle 90,thick]
node[left]{$\calE_3$} (v1)
(v1) edge[->,>=angle 90,thick]
node[right]{$\calE_4$} (v4)
(v4) edge[->,>=angle 90,thick]
node[left]{$\calE_5$} (v3)
(v1) edge[->,>=angle 90,thick]
node[above]{$\calE_6$} (v5)
(v5) edge[->,>=angle 90,thick]
node[left]{$\calE_7$} (v2)
(v4) edge[dashed,->,>=angle 90,thick]
node[near end, above]{$d_1$} (4cm, 2cm)
(0cm,4.5cm) edge[dashed,->,>=angle 90,thick]
node[near start,left]{$d_1$} (v5)
(-1cm,3.5cm) edge[dashed,->,>=angle 90,thick]
node[near start,below]{$d_2$} (v5)
(v3) edge[dashed,->,>=angle 90,thick]
node[near end,above]{$d_2$} (2.5cm,0.5cm);
\end{tikzpicture}
\caption{Distribution network \eqref{e:plant} on the graph. The state $x_i$ is the storage at the vertex $v_i$. The flows on the corresponding edges are denoted as $u_i$. The orientations on the edges are consistent with the incidence matrix. The vertices $v_5$ has inflow $d_1$ and $d_2$, the vertices $v_4$ and $v_3$ have outflow $d_1$ and $d_2$, respectively. In this case, $E\in \R^{5\times 2}$ whose first and second column are $[0,0,0,-1,1]^\top$ and $[0,0,-1,0,1]^\top$, respectively.}\label{figure main ex1}
\end{center}
\end{figure}



In this paper, we consider the flows on the edges are proportional to the state differences of the adjacent vertices. More precisely, the flows are given as 
\begin{align}\label{e:controller}
u = W B^\top x,
\end{align}
where the diagonal matrix $W\in\R^{m\times m}$ is the control gain. The output $y\in\R^k$, which measures the state difference at each port, is given as
\begin{align}\label{e:output}
y= & E^Tx.
\end{align}
This form of the output can be due to the physical constraints of the distribution network, i.e., only the state differences at the ports can be measured. Furthermore, for SISO dynamical distribution networks defined on some special graphs, it can be shown that the induced $\calL_2$ gain from $d$ to $y$ in \eqref{e:output} is the largest among all $y =C x$ with $C\in\R^{1\times n}$ and  $C\mathds{1}=0$. See Corollary \ref{cor:complete} in appendix for details.
Now the closed-loop is, 
\begin{equation}\label{flowsys}
\begin{split}
 \dot{x}= & -L_wx+Ed,  \\
 y= & E^Tx,
 \end{split}
\end{equation}
where $L_w$ is the graph Laplacian of $\calG=(\calW,\calV,\calE)$ and $\calW$ is the set of weights specifying by control gain $W$ in \eqref{e:controller}.

We are ready to introduce two problems which we shall tackle in this paper.

\textbf{Optimal Weight Allocation:} For a given graph and a positive constant $c$, 
\begin{align}\label{oriopt}
\min_{W}  &\|\mathbb{G}\|_\infty  \\ \nonumber
s.t.,    & \sum w_i =c, \ w_i \geq 0,
    \end{align}
where $\mathbb{G}$ is the transfer function (from $d$ to $y$) of the system (\ref{flowsys}), $W=~\diag(w_1,...,w_m)$ and $w_i$, for $i=1,...,m$, are the weights on the edges, and $c$ is a positive constant. 


\textbf{Positive semi-definite Laplacian:} Given a weighted graph $\calG$ with both positive and negative edge weights, what are the upper bounds on the magnitudes
of the negative weights in order to have the Laplacian to be positive semi-definite?

%

The following two sections are devoted to these two problems, respectively. Before proceeding, this section is closed with following physical interpretation of distribution networks.

\begin{example} 
One physical interpretation of the system (\ref{flowsys}) is a basic model of a dynamic flow network, where there are water reservoirs on the vertices and pipes on the edges. The reservoirs are identical cylinders and the pipes are horizontal. The state $x$ is constituted by the water levels in the reservoirs and the pressures are proportional to the water levels. The flow in the pipes are passively driven by pressure difference between the reservoirs. The weights $ \mathcal{W}$ are representing the capacities of the pipes, in terms of diameter and friction. The passive flow from reservoir $i$ to reservoir $j$ is then $q_{ij}=~w_{ij}(x_i-x_j)$. The external input $d$ can e.g. be interpreted as flow in pumps which are distributing water inside the network. The output $y$ is then the difference between water levels of the reservoirs which the pumps are pumping to and the reservoirs which the pumps are pumping from.


\end{example}

\medskip 

\section{$\mathcal{H}_\infty$-norm of the distribution network}\label{distnet}

In this section, we shall solve the optimal weight allocation problem by reformulating problem \eqref{oriopt} as an equivalent optimization problem with LMIs as constraints, which can then be efficiently solved numerically using, e.g.,  CVX \cite{cvx}. The main result of this section is presented as follows.

\begin{theorem}\label{thm:main}
Consider the system (\ref{flowsys}), where $\mathcal{G}$ is an undirected graph and each port belongs to exactly one connected component of $\mathcal{G}$. Suppose $L_w \succcurlyeq 0$. Then 
\begin{enumerate}
\item the $\calH_\infty \text{- norm}$ of  (\ref{flowsys}) is finite, and
\item the following statements are eqvivalent:
\begin{itemize}
  \item  the $\calH_\infty \text{- norm}$ is less than or equal to $\gamma$.
 \item the following LMI is satisfied,
\begin{equation}\label{newLMI}
\begin{bmatrix}
L_w & E \\
E^\top & \gamma I_k
\end{bmatrix}\succcurlyeq  0. 
\end{equation}
\end{itemize}
\end{enumerate}
\end{theorem}

\begin{proof}
We show that the theorem is true for the case where $\mathcal{G}$ has exactly one connected component. This is without loss of generality since if $\mathcal{G}$ has more than one connected component, the same procedure can be done for each component and the LMIs can be merged with a common $\gamma$. Denote 
\begin{align}\label{e:U}
U^\top=~[\frac{1}{\sqrt{n}}\mathds{1}_n, u^\top_2, \ldots, u^\top_n] \ \text{and} \ U_2^\top~=~[u^\top_2, \ldots, u^\top_n],
\end{align}
for which 
$U L_w U^\top =~\diag(0, \lambda_2, \ldots,\lambda_n)=:~\Lambda$. Denote $\hat{\Lambda}=~\diag(\lambda_2,\ldots,\lambda_n)$. Then the system (\ref{flowsys}) has equal $\calH_\infty \text{- norm}$ as the system
\begin{align*}
\dot{\tilde{x}} & = - \Lambda \tilde{x} + UE d, \\
z & = E^\top U^\top \tilde{x}.
\end{align*}
Notice that the first row of $UE$ is zero, thus the $\calH_\infty \text{- norm}$ of the system (\ref{flowsys}) equals the $\calH_\infty \text{- norm}$ of the system
\begin{equation}\label{e:flow_P_mini}
\begin{aligned}
\dot{\hat{x}} & = - \hat{\Lambda} \hat{x} + U_2E d, \\
z & = E^\top U_2^\top \hat{x}.
\end{aligned}
\end{equation} 
Due to the symmetry of the system and by Theorem 6 in \cite{TAN2001}, the $\calH_\infty \text{- norm}$ of the system  (\ref{e:flow_P_mini}) is $\|E^\top U_2^\top \hat{\Lambda}^{-1} U_2E \|_{2}$, which is finite. The $\calH_\infty \text{- norm}$ of the system \eqref{flowsys} is then less than or equal to $\gamma$ if and only if
\begin{equation} \nonumber
\|E^\top U_2^\top \hat{\Lambda}^{-1} U_2E \|_{2} \leq \gamma.
\end{equation}
By the property of real symmetric matrix, we can further rewrite the previous constrain as $E^\top U_2^\top \hat{\Lambda}^{-1} U_2E \preccurlyeq \gamma I_{k}$. By Schur complement, we have
\begin{align} \nonumber
\begin{bmatrix}
\hat{\Lambda} & U_2 E \\
E^\top U_2^\top & \gamma I_{k}
\end{bmatrix} \succcurlyeq 0,
\end{align}
which is equivalent to 
\begin{equation} \nonumber
\begin{bmatrix}
\Lambda & UE \\
E^\top U^\top & \gamma I_{k}
\end{bmatrix} \succcurlyeq 0. 
\end{equation}
By pre and post multiplication of matrix $\diag(U^\top, I_k)$ and $\diag(U, I_k)$, respectively, the previous inequality is transformed to
\begin{equation} \nonumber
\begin{bmatrix}
L_w & E \\
E^\top & \gamma I_k
\end{bmatrix} \succcurlyeq 0.
\end{equation}
Then the conclusion follows.
\end{proof}

\begin{remark}
Notice that in Theorem \ref{thm:main}, the weighted Laplacians $L_w$ can have both positive and negative weights. The result still holds as long as $L_w$ is positive semi-definite. 
\end{remark}

\begin{remark}
By  Theorem \ref{thm:main}, the problem (\ref{oriopt}) is equivalent to the following semi-definite programming (SDP) problem
\begin{equation}\label{optnew}
\begin{aligned}
\min_W & \quad \gamma \\
s.t., & \ \ 
\begin{bmatrix}
L_w & E \\
E^\top & \gamma I_k
\end{bmatrix} \succcurlyeq 0, \\
& \ \ \sum w_i = c, \ w_i \geq 0,
\end{aligned}
\end{equation}
which can be efficiently solved by e.g., CVX.

As one numerical example, we consider problem \eqref{optnew} defined on the graph in Fig. \ref{figure main ex1} with $c=8$. Then the optimal weights are $w_1=0, w_2 =1.0427, w_3 = 2, w_4 = 3.0427, w_5 = 0.9573, w_6 = 0.9573,$ and $w_7 = 0$. Here the minimum is $\gamma = 1$. It can be seen that the flows on the first and seventh edge do not contribute to the minimization of $\calH_{\infty} \text{- norm}$ of this network. The mechanism of this weight allocation is under investigation. 
\end{remark}

It is worth mentioning that in a recent work \cite{Pirani2017}, the authors considered a $\calH_\infty$ design problem for system \eqref{flowsys} with grounded Laplacian with respect to the topology, instead of weight allocation.

In Theorem \ref{thm:main}, we proved that the inequality \eqref{newLMI} is satisfied if and only if the $\calH_\infty \text{- norm}$ is less than or equal to $\gamma$. Moreover, by the bounded real lemma for state-space symmetric systems, i.e., Lemma \ref{lm:bounded-real-ssss}, we have that the following two statements are equivalent
\begin{itemize}
\item $\|\mathbb{G}\|_\infty \leq \gamma$,
\item the Riccati inequality 
\begin{equation}\label{flow2feb}
 -PL_w-L_w^TP+EE^T+ \frac{1}{\gamma^2}PEE^TP \preccurlyeq 0.
\end{equation}
 is satisfied with the solution $P=\gamma I$.
\end{itemize}
In this case, \eqref{flow2feb} is simplified as 
\begin{equation}\label{LMI94}
-L_w+\frac{EE^T}{\gamma}\preccurlyeq 0.
\end{equation}

In next section, we shall focus on the positive semi-definiteness of weighted Laplacian, which is a key assumption in Theorem \ref{thm:main}. It turns out that the inequality \eqref{LMI94} plays a crucial role.

\medskip

\section{Positive Semidefiniteness of Signed Laplacians}\label{s:Positive-Laplacian}

In this section, we consider the positive semidefiniteness of signed Laplacian matrices. The main result of this section is formulated in the following theorem, we establish the relation between the magnitude of the negative weights and the effect resistance matrix of subgraph $\calG_+$. In \cite{Zelazo2014}, the authors assumed that for any $(i,j)\in\calE_-$ and $(i',j')\in\calE_-$ being two distinct pairs of vertices, there is no cycle in $\calG_+$ containing $i,j,i'$ and $j'$.  Here we relax the condition to general graphs.

\begin{theorem}\label{th:iff-semip}
The Laplacian matrix $L$ is positive semidefinite if and only if 
\begin{enumerate}
\item for any $e_- = (i,j)\in\calE_-$, $i,j$ belong to one connected component of $\calG_+$, and 
\item the magnitude of the negative weights satisfies
\begin{align}\label{e:iff-semipositive}
\calW_-^{-1} \succcurlyeq B_-^\top L_{w+}^\dagger B_-.
\end{align} 
\end{enumerate}
\end{theorem}

\begin{proof}
\emph{Sufficiency:}
Since for any $e_- = (v_i,v_j)\in\calE_-$, $i,j$ belong to one connected component of $\calG_+$, by Theorem \ref{thm:main}, we have the system 
\begin{equation}\label{e:sys-semipositive-Laplacian}
\begin{aligned}
\dot{x} & = -L_{w+} x + B_-\sqrt{\calW_-} d \\
y & = \sqrt{\calW_-} B^\top_- x
\end{aligned}
\end{equation}
has finite $\calH_\infty$ norm. Furthermore, by Lemma \ref{lm:bounded-real-ssss}, we have that $L_w = L_{w+} -B_-\sqrt{\calW_-}\sqrt{\calW_-} B^\top_- \succcurlyeq 0$ is equivalent to the induced $\calL_2$ gain of system  \eqref{e:sys-semipositive-Laplacian} is less than or equal to $1$.  

To prove that $\calW_-^{-1} \succcurlyeq B_-^\top L_{w+}^\dagger B_-$ implies $L_w\succcurlyeq 0$, we only focus on the case that $\calG_+$ is connected, i.e., there is only one connected component, without loss of generality. Since 
\begin{align*}
 \sqrt{\calW_-}B_-^\top L_{w+}^\dagger B_-\sqrt{\calW_-} 
=  \sqrt{\calW_-}B_-^\top U^\top_2 \Lambda_+^{\dagger} U_2  B_-\sqrt{\calW_-},
\end{align*}
where $U_2$ is given as in \eqref{e:U} such that $U_2 L_{w+} U^\top_2 = \Lambda_+$, and the induced $\calL_2$ gain from $d$ to $y$ of system \eqref{e:sys-semipositive-Laplacian} is $\|\sqrt{\calW_-}B_-^\top U^\top_2 \hat{\Lambda}_+^{-1} U_2  B_-\sqrt{\calW_-}\|_2$, we have
\begin{align*}
& \calW_-^{-1} \succcurlyeq B_-^\top L_{w+}^\dagger B_-  \\
\Longleftrightarrow  \quad & \|\sqrt{\calW_-}B_-^\top L_{w+}^\dagger B_-\sqrt{\calW_-} \|_2\leq 1 \\
\Longleftrightarrow  \quad & \|\sqrt{\calW_-}B_-^\top U^\top_2 \hat{\Lambda}_+^{-1} U_2  B_-\sqrt{\calW_-}\|_2 \leq 1.
\end{align*}
Then the conclusion follows.

\emph{Necessity:} First, it can be verified that if there exists an edge $e_- = (v_i,v_j)\in\calE_-$ such that $v_i,v_j$ belong to different connected components of $\calG_+$, $L$ can not be positive semidefinite. More precisely, suppose $\calG_+$ has $N$ connected components, and the vertices set $\calV$ can be divided as $\calV=\calV_1\cup \cdots \cup \calV_N$ with $|\calV_i| = n_i$ and $\sum_{i=1}^{N}n_i = n$. Furthermore, w.l.o.g., suppose $e_-=(v_i,v_j)\in\calE_-$ such that $v_i$ and $v_j$ belongs to the first and second component, respectively. Denote the Laplacian of the graph $(\calV,\calE\setminus\{(i,j)\})$ as $\tilde{L}_{w-}$. Then by choosing $v^\top = (\mathds{1}_{n_1}, \frac{1}{2}\mathds{1}_{n_2}, 0,\ldots,0)^\top$, we have 
\begin{align*}
v^\top L_w v \leq -\frac{1}{4} \calW_{-,ij} <0
\end{align*} 
where $\calW_{-,ij}>0$ is the magnitude of the negative weights of $(v_i,v_j)$, which is contradict to the positive semi-definiteness of $L_w$.

With the item 1) holding, the necessity of \eqref{e:iff-semipositive} follows directly from the sufficiency part of the proof.
\end{proof}

\begin{remark}
Notice that the matrix $B_-^\top L_{w+}^\dagger B_-$ is a submatrix of the effective resistance matrix of $\calG_+$.
When there is only one negative edge, the condition \eqref{e:iff-semipositive} is equivalent to Theorem III.3 in \cite{Zelazo2014}. However, for the multiple negative edges, the result in Theorem \ref{th:iff-semip} is more general than Theorem III.4 in \cite{Zelazo2014} in the sense that there are no constraints on the positions of negative edges. 
\end{remark}

The intuition of Theorem \ref{th:iff-semip} is illustrated in the following example.

\begin{example}\label{ex:equivalence-posi-distri}
Consider a weighted graph with two negative edges given as in Fig. \ref{figure main2}. Suppose the negative weights of $(v_5,v_3)$ and $(v_5,v_4)$ are $-w_8$ and $-w_9$, respectively. Recall that the network in Fig.\ref{figure main ex1} represents a dynamical distribution network with two ports and in/outflows $d_1$ and $d_2$, respectively, and only positive edge weights. By setting 
\begin{align}\label{e:E}
E^\top = \begin{bmatrix}
 0 & 0 & 0 & -\sqrt{w_9} & \sqrt{w_9}\\
 0 & 0 & -\sqrt{w_8} & 0 & \sqrt{w_8}
\end{bmatrix}
\end{align} 
and by \eqref{LMI94}, we have the Laplacian of the graph in Fig. \ref{figure main2} is positive semi-definite, if and only if the dynamical distribution network in Fig. \ref{figure main ex1} with $E$ defined as \eqref{e:E} has $\calH_\infty$-norm no larger than $1$.

As an numerical example, we consider the positive weights of the graph in Fig. \ref{figure main2} are identical to one. In this case the submatrix of the effective resistance matrix
\begin{equation*}
    B_-^\top L_{w+}^\dagger B_-=\begin{bmatrix}
    1.1429  &  0.7143 \\
    0.7143  &  0.9048
\end{bmatrix}.
\end{equation*}
It can be verified that by choosing $w_8 = w_9 = 0.5$, we have that \eqref{e:iff-semipositive} holds. In this case, the eigenvalues of $L_w$ are $0, 0.2, 2.6, 4.2, 5$. However, by choosing $w_8 = 0.7, w_9 = 0.5$, which violates \eqref{e:iff-semipositive}, the eigenvalues of $L_w$ are $-0.04, 0, 2.4, 4.2, 5$.

\end{example}

\begin{figure}[ht]
\centering
\begin{tikzpicture}
\tikzstyle{EdgeStyle}    = [thin,double= red,
                            double distance = 0.0pt]
\useasboundingbox (0,0) rectangle (4cm,4cm);
\tikzstyle{VertexStyle} = [shading         = ball,
                           ball color      = white!100!white,
                           minimum size = 20pt,%
                           inner sep       = 1pt,]
\Vertex[style={minimum
size=0.2cm,shape=circle},LabelOut=false,L=\hbox{$v_1$},x=1.5cm,y=3.5cm]{v1}
\Vertex[style={minimum
size=0.2cm,shape=circle},LabelOut=false,L=\hbox{$v_2$},x=0cm,y=2cm]{v2}
\Vertex[style={minimum
size=0.2cm,shape=circle},LabelOut=false,L=\hbox{$v_3$},x=1.5cm,y=0.5cm]{v3}
\Vertex[style={minimum
size=0.2cm,shape=circle},LabelOut=false,L=\hbox{$v_4$},x=3cm,y=2cm]{v4}
\Vertex[style={minimum
size=0.2cm,shape=circle},LabelOut=false,L=\hbox{$v_5$},x=0cm,y=3.5cm]{v5}

\draw
(v5) edge[dashed, -,>=angle 90,thick, red]
node[right]{} (v4)
(v2) edge[-,>=angle 90,thick]
node[left]{} (v3)
(v3) edge[dashed,-,>=angle 90,thick, red]
node[left]{} (v5)
(v3) edge[-,>=angle 90,thick]
node[left]{} (v1)
(v2) edge[-,>=angle 90,thick]
node[left]{} (v1)
(v1) edge[-,>=angle 90,thick]
node[right]{} (v4)
(v4) edge[-,>=angle 90,thick]
node[left]{} (v3)
(v1) edge[-,>=angle 90,thick]
node[above]{} (v5)
(v5) edge[-,>=angle 90,thick]
node[left]{} (v2);
\end{tikzpicture}
\caption{ The network used in Example \ref{ex:equivalence-posi-distri} where the graph has two negative edges $(v_1,v_2)$ and $(v_1,v_3)$ (red colored).}\label{figure main2}
\end{figure}
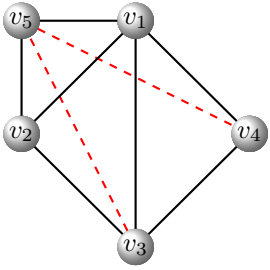

\section{Conclusions}\label{Concl}

For a basic dynamic distribution networks, we have derived an optimization set up with LMIs as constraints, which minimizes the $\calH_\infty \text{- norm}$ with respect to the allocation of the weights on the edges. Furthermore, by using a bounded real lemma for state-space symmetric systems, we have interpreted the Riccati inequality for distribution networks as a definiteness criterion of a Laplacian to a graph containing both positive and negative weights on the edges. Moreover, we have provided a sufficient and necessary condition, using effective resistance matrix of the subgraph spanned the positive edges, for the positive semi-definiteness of the Laplacian with both positive and negative edges. 

A related future topic is the problem of minimizing the $\calH_\infty \text{- norm}$ of dynamic flow networks with respect to topology, more precisely, a limited amount of edges is to be allocated in a graph with fixed vertices. Another future topic is to consider a fixed graph (both topology and weights), but consider saturation of the flow on the edges. The problem is then to minimize the induced $\calL_2$-gain with respect to allocation of the saturation limits.

\section{Acknowledgment}

The first author would like to acknowledge Dr. Mohammad Pirani for the valuable discussions.

\section{Appendix}

\begin{proof}[Proof of Lemma \ref{lm:bounded-real-ssss}]
Notice that $\textit{(2)}$ implies $\textit{(1)}$ is guaranteed by the bounded real lemma. Hence we only show that $\textit{(1)}$ implies $\textit{(2)}$. 

By using Theorem 6 in \cite{TAN2001}, we have that $\|\mathbb{G}\|_\infty = \|-BA^{-1}B\|_2$. Hence, by Schur complement, $\|\mathbb{G}\|_\infty \leq \gamma$ implies 
\begin{align*}
\begin{bmatrix}
-A & B \\
B^\top & \gamma I
\end{bmatrix} \succcurlyeq 0,
\end{align*}
which is equivalent to $A+\frac{1}{\gamma}BB^\top \preccurlyeq 0$ and $\gamma\geq 0$. Then it is straightforward to see that $P=\gamma I$ is a solution to \eqref{e:Hinf-Riccati}.
\end{proof}

\begin{corollary}\label{cor:complete}
Consider a SISO dynamical distribution network
\begin{equation}\label{e:system-general-output}
\begin{aligned}
\dot{x} & = - L_w x + Ed, \\
y & = C x,
\end{aligned} 
\end{equation}
defined on a connected graph, with $E\in\R^{n\times 1}$ and $C\in\R^{1\times n}$ satisfying $\mathds{1}^\top E = C\mathds{1}=0$, then the $H_\infty-$norm is upper bounded by 
\begin{align}
\frac{\|C\|_2\|B\|_2}{\lambda_2}.
\end{align} 
Furthermore, suppose that the eigenvalues of $L_w$ satisfy $0=\lambda_1<\lambda_2=\cdots=\lambda_n$, then we have
\begin{align*}
E^\top = \arg \max_{C} & \quad \|\mathbb{G}\|_\infty, \\
s.t.  & \quad \|C\|_2 = \|E\|_2.
\end{align*}
\end{corollary}

\begin{proof}
The notations in this proof are consistent with the ones in the proof of Theorem \ref{thm:main}. 

The $H_\infty$ norm of system \eqref{e:system-general-output} is 
\begin{align*}
& \sup_{\omega\in\R} |C U_2^\top (j\omega I + \hat{\Lambda})^{-1} U_2E| \\
= & \sup_{\omega\in\R} | \sum_{i=1}^{n-1} \frac{\hat{C}_i \hat{E}_i}{j\omega+\lambda_{i+1}} | \\
\leq & \sup_{\omega\in\R} \sum_{i=1}^{n} \frac{|\hat{C}_i \hat{E}_i|}{\sqrt{\omega^2+\lambda^2_{i+1}}} \\
\leq &  \frac{\|C\|_2 \|E\|_2}{\lambda_2}
\end{align*}
where $\hat{C}_i$ and $\hat{E}_i$ are the $i$th components of the vectors $CU_2^\top $ and $U_2E$, respectively, and the last inequality is based on the fact that $\|CU_2^\top\|_2 = \|C\|_2$ and $\|U_2 E\|_2 = \|E\|_2$.

If we further have $\lambda_2=\cdots=\lambda_n$, i.e., the previous upper bound can be achieved if and only if $U_2 C^\top = U_2 E$. Then since $U_2^\top U_2 = I - \frac{1}{n}\mathds{1}\mathds{1}^\top$ and $\mathds{1}^\top E = C\mathds{1}=0$, we have $C= E^\top$. Thus the conclusion follows.
\end{proof}




\bibliographystyle{plain} 
\bibliography{ref}

\begin{thebibliography}{10}

\bibitem{schaftSIAM}
{A.J. van~der Schaft and B.M. Maschke}.
\newblock Port-{H}amiltonian systems on graphs.
\newblock {\em SIAM J. Control and Optimization}, 51(2):906--937, 2013.

\bibitem{Altafini2013}
C.~Altafini.
\newblock Consensus problems on networks with antagonistic interactions.
\newblock {\em IEEE Transactions on Automatic Control}, 58(4):935--946, 2013.

\bibitem{Aronson1989}
J.~Aronson.
\newblock A survey of dynamic network flows.
\newblock {\em Annals of Operations Research}, 20(1):1--66, 1989.

\bibitem{Moreno1995}
J.C.~Moreno Banos and M.~Papageorgiou.
\newblock A linear programming approach to large-scale linear optimal control
  problems.
\newblock {\em IEEE Transactions on Automatic Control}, 40(5):971--977, 1995.

\bibitem{Bertsimas2006}
D.~Bertsimas and A.~Thiele.
\newblock A robust optimization approach to inventory theory.
\newblock {\em Operations Research}, 54(1):150--168, 2006.

\bibitem{Blanchini00}
F.~Blanchini, S.Miani, and W.Ukovich.
\newblock Control of production-distribution systems with unknown inputs and
  system failures.
\newblock {\em Automatic Control, IEEE Transactions on}, 45(6):1072--1081,
  2000.

\bibitem{Bollobas98}
B.~Bollobas.
\newblock {\em Modern Graph Theory}, volume 184 of {\em Graduate Texts in
  Mathematics}.
\newblock Springer, New York, 1998.

\bibitem{ChenW2016}
W.~Chen, J.~Liu, Y.~Chen, S.~Z. Khong, D.~Wang, T.~Ba\c{s}ar, L.~Qiu, and K.~H.
  Johansson.
\newblock Characterizing the positive semidefiniteness of signed laplacians via
  effective resistances.
\newblock In {\em 55th IEEE Conference on Decision and Control}, pages
  985--990, 2016.

\bibitem{ChenW2017}
W.~Chen, D.~Wang, J.~Liu, T.~Ba\c{s}ar, and L.~Qiu.
\newblock On spectral properties of signed laplacians for undirected graphs.
\newblock In {\em 56th IEEE Conference on Decision and Control}, pages
  1999--2002, 2017.

\bibitem{ChenY2016}
Y.~Chen, S.~Z. Khong, and T.~T. Georgiou.
\newblock On the definiteness of graph laplacians with negative weights:
  Geometrical and passivity-based approaches.
\newblock In {\em Proceedings of the American Control Conference}, pages
  2488--2493, 2016.

\bibitem{Como2015}
G.~Como, E.~Lovisari, and K.~Savla.
\newblock Throughput optimality and overload behavior of dynamical flow
  networks under monotone distributed routing.
\newblock {\em IEEE Transactions on Control of Network Systems}, 2(1):57--67,
  2015.

\bibitem{DANIELSON2013}
C.~Danielson, F.~Borrelli, D.~Oliver, D.~Anderson, and T.~Phillips.
\newblock Constrained flow control in storage networks: capacity maximization
  and balancing.
\newblock {\em Automatica}, 49(9):2612 -- 2621, 2013.

\bibitem{bullo2014}
F.~D\"{o}rfler and F.~Bullo.
\newblock Kron reduction of graphs with applications to electrical networks.
\newblock {\em IEEE Transactions on Circuits and Systems I: Regular Papers},
  60(1):150--163, Jan 2013.

\bibitem{Ephremides1989}
A.~Ephremides and S.~Verd\'{u}.
\newblock Control and optimization methods in communication network problems.
\newblock {\em IEEE Transactions on Automatic Control}, 34(9):930--942, 1989.

\bibitem{Ford1956}
L.~R. Ford and D.~R. Fulkerson.
\newblock Maximal flow through a network.
\newblock {\em Canadian Journal of Mathematics}, 8:399--404.

\bibitem{cvx}
M.~Grant and S.~Boyd.
\newblock {CVX}: Matlab software for disciplined convex programming, version
  2.1.
\newblock \url{http://cvxr.com/cvx}, March 2014.

\bibitem{wei2013}
{J. Wei and A.J. van der Schaft}.
\newblock Load balancing of dynamical distribution networks with flow
  constraints and unknown in/outflows.
\newblock {\em Systems \& Control Letters}, 62(11):1001--1008, 2013.

\bibitem{Klein93}
D.~J. Klein and M.~Randi\'{c}.
\newblock Resistance distance.
\newblock {\em Journal of Mathematical Chemistry}, 12(1):81–95, 1993.

\bibitem{Lovisari2014}
E.~Lovisari, G.~Como, and K.~Savla.
\newblock Stability of monotone dynamical flow networks.
\newblock In {\em 53rd IEEE Conference on Decision and Control}, pages
  2384--2389, 2014.

\bibitem{Moss1982}
F.H. Moss and A.~Segall.
\newblock An optimal control approach to dynamic routing in networks.
\newblock {\em IEEE Transactions on Automatic Control}, 27(2):329--339, 1982.

\bibitem{Nagashio2005}
T.~Nagashio and T.~Kida.
\newblock Symmetric controller design for symmetric plant using matrix
  inequality conditions.
\newblock In {\em Proceedings of the 44th IEEE Conference on Decision and
  Control}, pages 7704--7707, Dec 2005.

\bibitem{Pirani2017}
M.~Pirani, E.~Moradi Shahrivar, B.~Fidan, and S.~Sundaram.
\newblock Robustness of leader-follower networked dynamical systems.
\newblock {\em IEEE Transactions on Control of Network Systems}, 2017.

\bibitem{Qiu1996}
L.~Qiu.
\newblock On the robustness of symmetric systems.
\newblock {\em Systems \& Control Letters}, 27(3):187 -- 190, 1996.

\bibitem{rai2012}
A.~Rai, D.~Ward, S.~Roy, and S.~Warnick.
\newblock Vulnerable links and secure architectures in the stabilization of
  networks of controlled dynamical systems.
\newblock In {\em 2012 American Control Conference (ACC)}, pages 1248--1253,
  June 2012.

\bibitem{RANTZER2015}
A.~Rantzer.
\newblock Scalable control of positive systems.
\newblock {\em European Journal of Control}, 24:72 -- 80, 2015.

\bibitem{rockafellar1984network}
R.T. Rockafellar.
\newblock {\em Network flows and monotropic optimization}.
\newblock Pure and applied mathematics. Wiley, 1984.

\bibitem{SCHOLTEN2017}
T.~Scholten, S.~Trip, and C.~De Persis.
\newblock Pressure regulation in large scale hydraulic networks with input
  constraints.
\newblock In {\em 20th IFAC World Congress}, volume~50, pages 5367 -- 5372,
  2017.

\bibitem{Segall1977}
A.~Segall.
\newblock The modeling of adaptive routing in data-communication networks.
\newblock {\em IEEE Transactions on Communications}, 25(1):85--95, 1977.

\bibitem{TAN2001}
K.~Tan and K.~M. Grigoriadis.
\newblock Stabilization and ${\calH}_{\infty}$- control of symmetric systems:
  an explicit solution.
\newblock {\em Systems \& Control Letters}, 44(1):57 -- 72, 2001.

\bibitem{Tanaka2011}
T.~Tanaka and C.~Langbort.
\newblock The bounded real lemma for internally positive systems and
  ${\calH}_{\infty}$ structured static state feedback.
\newblock {\em IEEE Transactions on Automatic Control}, 56(9):2218--2223, 2011.

\bibitem{WILLEMS1976}
J.~C. Willems.
\newblock Realization of systems with internal passivity and symmetry
  constraints.
\newblock {\em Journal of the Franklin Institute}, 301(6):605 -- 621, 1976.

\bibitem{Xia2016}
W.~Xia, M.~Cao, and K.~H. Johansson.
\newblock Structural balance and opinion separation in trust-mistrust social
  networks.
\newblock {\em IEEE Transactions on Control of Network Systems}, 3(1):46--56,
  2016.

\bibitem{Xiao2003}
L.~Xiao and S.~Boyd.
\newblock Fast linear iterations for distributed averaging.
\newblock In {\em 42nd IEEE Conference on Decision and Control}, volume~5,
  pages 4997--5002 Vol.5, 2003.

\bibitem{Yang2001}
G.~H. Yang, J.~L. Wang, and Y.~C. Soh.
\newblock Decentralized control of symmetric systems.
\newblock {\em Systems \& Control Letters}, 2001.

\bibitem{Zelazo2014}
D.~Zelazo and M.~B\"{u}rger.
\newblock On the definiteness of the weighted laplacian and its connection to
  effective resistance.
\newblock In {\em 53rd IEEE Conference on Decision and Control}, pages
  2895--2900, 2014.

\bibitem{zhou1998essentials}
K.~Zhou and J.C. Doyle.
\newblock {\em Essentials of Robust Control}.
\newblock Prentice Hall Modular Series f. Prentice Hall, 1998.

\end{thebibliography}


\end{document}